\documentclass{article}
\usepackage[utf8]{inputenc}
\usepackage[english]{babel}
\usepackage{graphicx}
\usepackage{amsmath}
\usepackage{amsthm}
\usepackage{amsfonts} 

\usepackage{biblatex}

\usepackage[bookmarks=false]{hyperref}

\newtheorem{theorem}{Theorem}
\usepackage{verbatim}
\usepackage{authblk}
\usepackage{xcolor} 
\definecolor{LightGray}{gray}{0.9}

\title{Learning (With) Distributed Optimization}
\author[1]{Aadharsh Aadhithya A, Abinesh S, Akshaya J, Jayanth M, Vishnu Radhakrishnan, Sowmya V, Soman K.P}
\affil[1]{Center for Computational Engineering and Networking, School of Computing, Amrita Vishwa Vidyapeetham, Coimbatore}

\date{ 27th January, 2023}

\begin{document}

\maketitle

\section{Introduction}
This article is meant to be an introductory tutorial for Distributed Optimization. The article is structured such that the reader gets enough context to quickly come up to pace with one of the latest Distributed Optimization algorithm: ALADIN.  The material is a by-product of the notes collected by authors while learning distributed optimization. Feel free to contact the authors in case of any mistakes encountered in the article. Have a good read ahead! 
( \href{https://slides.com/aie-24amrita/aladin/fullscreen}{Click here to access the teaching material(PPT) for the tutorial})
\subsection{History of Distributed optimization}
Early works on distributed optimization trace back to Everett Dantzig, Wolfe and Benders in the 1960s. These first works mainly considered Lagrangian relaxation for strictly convex problems and decomposition methods for linear programs. Later, the Lagrangian relaxation was combined with augmented Lagrangian techniques developed mainly by Hestenes, Powell and Miele to improve numerical stability and to provide guarantees also for convex but not strictly convex problems. This led to first versions of ADMM with improved convergence guarantees and improved practical convergence. Many of these duality-based works are summarized textbooks by Bertsekas and Tsitsiklis [BT89] and Censior and Zenios. \\
\par
Distributed optimization gained new interest in the late 2000s mainly in the field of machine learning and imaging science, where ADMM outperformed state-of-the art methods in certain applications. Moreover, new applications in signal recovery emerged. The main motivation here was computational speedup, i.e. to find methods for parallel computing. Moreover, state-of-the-art methods have been shown to be a special case of ADMM allowing their treatment in a unified framework. Duality-based optimization methods were used in communication networks beginning already in the late 1990s. Similar approaches were used in wireless sensor networks, in signal processing and in a few fields of machine learning like Support Vector Machines.
A different version of dual decomposition is established in called proximal center method. Herein—instead of minimizing the augmented Lagrangian in an alternating fashion to achieve separability the author adds two linear proximal terms which are separable and lead to a differentiable dual function. An early dual-decomposition based approach for a multi-agent setting was presented. The highly influential paper showed that many of the above works can be treated in a unified framework based on ADMM. The importance of this framework lies in its generality.

\section{PRE-REQUISITE}

\subsubsection{Gradient}
After learning that functions with a multidimensional input have partial derivatives, you might wonder what the full derivative of such a function is. In the case of scalar-valued multivariable functions, meaning those with a multidimensional input but a one-dimensional output, the answer is the gradient.

\begin{equation}
    \nabla f\left(x_1 ,x_2 ,\ldotp \ldotp \ldotp \ldotp \ldotp ,x_n \right)=\left\lbrack \begin{array}{c}
\frac{\partial f}{\partial x_1 }\\
\frac{\partial f}{\partial x_2 }\\
\ldotp \\
\ldotp \\
\ldotp \\
\frac{\partial f}{\partial x_2 }
\end{array}\right\rbrack
\end{equation}

\subsubsection{Jacobian }
Suppose \( f : \mathbb{R}^n \rightarrow \mathbb{R}^m \) is a function such that each of its first-order partial derivatives exists on \( \mathbb{R}^n \). This function takes a point \( x \in \mathbb{R}^n \) as input and produces the vector \( f(x) \in \mathbb{R}^m \) as output. Then the Jacobian matrix of \( f \) is defined to be an \( m \times n \) matrix, denoted by \( J \).

$\mathbf {J} ={\begin{bmatrix}{\dfrac {\partial \mathbf {f} }{\partial x_{1}}}&\cdots &{\dfrac {\partial \mathbf {f} }{\partial x_{n}}}\end{bmatrix}}={\begin{bmatrix}\nabla ^{\mathrm {T} }f_{1}\\\vdots \\\nabla ^{\mathrm {T} }f_{m}\end{bmatrix}}={\begin{bmatrix}{\dfrac {\partial f_{1}}{\partial x_{1}}}&\cdots &{\dfrac {\partial f_{1}}{\partial x_{n}}}\\\vdots &\ddots &\vdots \\{\dfrac {\partial f_{m}}{\partial x_{1}}}&\cdots &{\dfrac {\partial f_{m}}{\partial x_{n}}}\end{bmatrix}}$

\subsubsection{Hessian}
Suppose \(f:\mathbb{R}^{n} \to \mathbb{R}\) is a function taking as input a vector \(\mathbf{x} \in \mathbb{R}^{n}\) and outputting a scalar \(f(\mathbf{x}) \in \mathbb{R}\). If all second partial derivatives of \(f\) exist, then the Hessian matrix \(\mathbf{H}\) of \(f\) is a square \(n \times n\) matrix,
\[
\mathbf{H}_{f} = \begin{bmatrix}
    \dfrac{\partial^2 f}{\partial x_{1}^{2}} & \dfrac{\partial^2 f}{\partial x_{1} \, \partial x_{2}} & \cdots & \dfrac{\partial^2 f}{\partial x_{1} \, \partial x_{n}} \\
    \dfrac{\partial^2 f}{\partial x_{2} \, \partial x_{1}} & \dfrac{\partial^2 f}{\partial x_{2}^{2}} & \cdots & \dfrac{\partial^2 f}{\partial x_{2} \, \partial x_{n}} \\
    \vdots & \vdots & \ddots & \vdots \\
    \dfrac{\partial^2 f}{\partial x_{n} \, \partial x_{1}} & \dfrac{\partial^2 f}{\partial x_{n} \, \partial x_{2}} & \cdots & \dfrac{\partial^2 f}{\partial x_{n}^{2}}
\end{bmatrix}.
\]

\subsubsection{Positive Semidefiniteness}
A positive semidefinite matrix is defined as a symmetric matrix with non-negative eigenvalues. The original definition is that a matrix \(M \in L(V)\) is positive semidefinite if,
\begin{enumerate}
    \item \(M\) is symmetric, and
    \item \(v^T \cdot M \cdot v \geq 0\) for all \(v \in V\).
\end{enumerate}

\subsubsection{Taylor's Series}
The Taylor series of a function is an infinite sum of terms that are expressed in terms of the function's derivatives at a single point. For most common functions, the function and the sum of its Taylor series are equal near this point. Taylor polynomials are approximations of a function, which generally become better as \(n\) increases. Taylor's theorem gives quantitative estimates of the error introduced by the use of such approximations.

\begin{equation}
   {\displaystyle f(a)+{\frac {f'(a)}{1!}}(x-a)+{\frac {f''(a)}{2!}}(x-a)^{2}+{\frac {f'''(a)}{3!}}(x-a)^{3}+\cdots ,}
   \end{equation}
 In the more compact sigma notation, this can be written as 
 \begin{equation}
     {\displaystyle \sum _{n=0}^{\infty }{\frac {f^{(n)}(a)}{n!}}(x-a)^{n},}
 \end{equation}

\subsubsection{Newton's Raphson's Method}
Newton's method, also known as the Newton–Raphson method, is a root-finding algorithm which produces successively better approximations to the roots (or zeroes) of a real-valued function. The most basic version starts with a single-variable function \( f \) defined for a real variable \( x \), the function's derivative \( f' \), and an initial guess \( x_0 \) for a root of \( f \).

\begin{equation}
    {\displaystyle x_{1}=x_{0}-{\frac {f(x_{0})}{f'(x_{0})}}}
\end{equation}

\begin{equation}
    {\displaystyle x_{n+1}=x_{n}-{\frac {f(x_{n})}{f'(x_{n})}}}
\end{equation}
The process is repeated until a sufficiently precise value is reached
\subsection{Various Norms}
Suppose we’re trying to solve an optimization problem. This means that we are trying to find the best input that minimizes some output penalty.
Norms are a great choice for penalties because they assign a reasonable magnitude to each output. Let us look at the various norms in the following section
\subsubsection{L2-Norm}
The L² norm takes the sum of the squared values, taking the square root at the end. The L² norm is the same as a standard distance formula, which finds the shortest path from A to B. Denoted by $||L||_2$
\subsubsection{Lp-Norm}
If you start to see a pattern, then you might ask: Why stop at 2? We can create a norm for all $1 \leq p < \infty$ in this way too!. This is nothing but the Lp norm.
\subsubsection{Weighted norm}
For a given vector x, and a positive Semidefinite matrix $\Sigma$, weighted norm is defined as
$$
\sqrt{x^T \Sigma x}
$$
Any positve semidefinite matrix can be viewed as a kernel matrix. Hence the above operation can be viewed as a scaled kernel mapping. 

\subsection{Understanding of convex function}
In this section we shall recall the basics of convex optimization like what exactly is a convex set and what is a convex function. Before going into that we need to have a small idea about what an epigraph of a function refers to. Epigraph or super-graph of a function ${f:X\to [-\infty ,\infty ]}$ valued in the extended real numbers ${\displaystyle [-\infty ,\infty ]=\mathbb {R} \cup \{\pm \infty \}}$ is the set, denoted by ${ {epi}\;f,}$ of all points in the Cartesian product ${ X\times \mathbb {R} }$ lying on or above its graph.

\subsubsection{Convex set}
A set C is a convex if every point on the line segment connecting x and y other than the endpoints is inside the topological interior of C. A closed convex subset is strictly convex if and only if every one of its boundary points is an extreme point. This theorem can be stated mathematically as follows, A set $C \epsilon R^{n_x}$ is called a convex if line between any two points $(x_1,x_2)$ lies entirely in $C$ i.e.,
\begin{equation}
    \theta.x_1 + (1-\theta)x_2 \; \epsilon \; C\;\;\; \forall\;\;\; \theta\; \epsilon\; (0,1)
\end{equation}

\subsubsection{Convex function}
A function $f:\chi \Rightarrow R$ is said to be convex if and only if it's epigraph is a convex set. i.e., for all $(x_1,x_2)\; \epsilon \; \chi$ we can say that,

\begin{equation}
    f(\theta.x_1 + (1-\theta)x_2)\; \leq \; \theta.f(x_1) + (1-\theta).f(x_2) \;\;\; \forall \;\;\; \theta\; \epsilon\; (0,1)
\end{equation}

\section{Duality}

According to optimization theory, duality or the principle of duality states that any optimization problems can be viewed from either of the two perspectives, the primal form or the dual form. If the primal is a minimization problem then the dual is a maximization problem and vice-versa. Any feasible solution to the primal problem which is minimization problem, is at least as large as any feasible solution to the dual problem which will be a maximization problem. Therefore, the solution to the primal is an upper bound to the solution of the dual, and the solution of the dual is a lower bound to the solution of the primal. This fact is called weak duality. In general, the optimal values of the primal and dual problems need not be equal. Their difference is called the duality gap. For convex optimization problems, the duality gap is zero under a constraint qualification condition. This fact is called strong duality.

\subsection{Duality Gap}
Duality gap represents the difference between the values obtained on solving the primal form and the solution obtained on solving the dual for of the same optimization problem. If $d*$ is the optimal dual value and $p*$ is the optimal primal value, then the duality gap is equal to $p*-d*$. This value is always greater than or equal to 0 (for minimization problems). The duality gap is zero if and only if strong duality holds. Otherwise the gap is strictly positive and weak duality holds.

\section{Why Dual form?}
Solving an optimization problem in it's dual form is more preferable that solving it in it's dual form because, The dual problem is always a convex optimization problem, even if the problem is not convex in it's primal form.

\begin{theorem}
The dual problem is a convex optimization problem.
\end{theorem}

\begin{proof}
By definition the of duality, given the Lagrangian function of an optimization problem the dual form can be written as in equation \ref{WD1},

\begin{equation}
\label{WD1}
    g(u,v) = inf_x\;\;f(x)+\sum_{i=1}^m \lambda_i. h_i(x) +\sum_{i=1}^n \mu_i. l_i(x)
\end{equation}

where \( \lambda_i \) is the Lagrangian multiplier corresponding to the \( i \)-th inequality constraint \( h_i(x) > 0 \) and \( \mu_i \) is the Lagrangian multiplier corresponding to the \( i \)-th equality constraint \( l_i(x) = 0 \). Here, the equation \ref{WD1} can be viewed as the point-wise infimum of affine functions of \( u \) and \( v \), thus it is concave. \( u \geq 0 \) represents affine constraints. Hence, the dual problem is a concave maximization problem, which is a convex optimization problem.

\end{proof}

\section{When Duality holds?}
Slater's condition (or Slater condition) is a sufficient condition for strong duality to hold for a convex optimization problem, named after Morton L. Slater. Informally, Slater's condition states that the feasible region must have an interior point. The interior of a subset S of a topological space X is the union of all subsets of S that are open in X. A point that is in the interior of S is an interior point of S.

\subsection{Formulation of Slater's theorem}
Consider the optimization problem,

\begin{equation}
\begin{array}{l}
\mathrm{Minimize}\;\;f_0 \left(x\right)\\
\mathrm{subject}\;\mathrm{to}:\\
\;\;\;\;\;\;\;\;f_i \left(x\right)\le 0,i=1,\ldotp \ldotp \ldotp \ldotp \ldotp ,m\\
\;\;\;\;\;\;\;A\ldotp x=b
\end{array}
\end{equation}

where ${\displaystyle f_{0},\ldots ,f_{m}}$ are convex functions. This is an instance of convex programming. In words, Slater's condition for convex programming states that strong duality holds if there exists an ${\displaystyle x^{*}}$ such that ${\displaystyle x^{*}}$ is strictly feasible (i.e. all constraints are satisfied and the nonlinear constraints are satisfied with strict inequalities)
Mathematically, Slater's condition states that strong duality holds if there exists an ${\displaystyle x^{*}\in \operatorname {relint} (D)}$ (where relint denotes the relative interior of the convex set ${\displaystyle D:=\cap _{i=0}^{m}\operatorname {dom} (f_{i})}$ such that

\begin{equation}
\begin{array}{l}
f_i \left(x^* \right)<0\;,\;i=1,\ldotp \ldotp \ldotp \ldotp \ldotp ,m\\
A\ldotp x^* =b
\end{array}
\end{equation}

\subsection{General form of Slater's theorem}
Given an optimization problem,

\begin{equation}
\begin{array}{l}
\mathrm{Minimize}\;\;f_0 \left(x\right)\\
\mathrm{subject}\;\mathrm{to}:\\
\;\;\;\;\;\;\;\;f_i \left(x\right)\le k_i ,i=1,\ldotp \ldotp \ldotp \ldotp \ldotp ,m\\
\;\;\;\;\;\;\;A\ldotp x=b
\end{array}
\end{equation}

where $f_{0}$ is convex and $f_{i}$ is $K_{i}$-convex for each $i$. Then Slater's condition says that if there exists an ${\displaystyle x^{*}\in \operatorname {relint} (D)}$ such that

\begin{equation}
\begin{array}{l}
f_i \left(x^* \right)<k_i \;,\;i=1,\ldotp \ldotp \ldotp \ldotp \ldotp ,m\\
A\ldotp x^* =b
\end{array}
\end{equation}

\section{Dual Ascent}

To understand the algorithmic steps of the Dual Ascent method of optimization let us consider a convex minimization problems with equality constraint as in equation \ref{DAP},

\begin{equation}
\begin{array}{l}
\min \;f\left(x\right)\\
\mathrm{subject}\;\mathrm{to}\;\mathrm{Ax}=b
\end{array}
\label{DAP}
\end{equation}

The Lagrangian function of this minimization problem can be written as mentioned in equation \ref{DAL},

\begin{equation}
L\left(x,\lambda \right)=f\left(x\right)+\lambda^T \left(\mathrm{Ax}-b\right)
\label{DAL}
\end{equation}

 And the dual form of this can be written as mentioned in equation \ref{DAD},

\begin{equation}
g\left(\lambda\right)=\inf_x \;L\left(x,\lambda\right)
\label{DAD}
\end{equation}

Here, the $\lambda$ represents the Lagrangian coefficients of the Primal form. On, writing the optimization problem in dual form we shall observe that these variables have become the primary variables for optimization. Here, inf represents the infinum value of the function. Infinum of a subset `S' of a partially ordered set `P', is a greatest element of P, that is less than or equal to each element of `S', if such an element exists. It is commonly termed as the greatest lower bound of that set.\\

So, the solution which we are supposed obtain at the end of Dual Ascent optimization is as in the following equation,

\begin{equation}
    y= arg max_y\;\; g(y)
\end{equation}

The dual ascent algorithm solves in an iterative manner in order to converge towards this solution using the following steps:

\begin{enumerate}
    \item Initiate `y' as a randomly generated value
    \item find the next iteration value of `x', by solving for g(y) as per the following equation,
    
    \begin{equation}
        x^{k+1} ={\mathrm{argmin}}_x \;L\left(x,y^k \right)
    \end{equation}
    
    \item find the next iteration value of `y' using the gradient descent step in the following equation,
    
    \begin{equation}
        y^{k+1} =y^k +a^k \left({\mathrm{Ax}}^{k+1} -b\right)
        \label{DAY}
    \end{equation}
\end{enumerate}

The main disadvantage of dual ascent is that there are few restrictions in order to apply dual ascent. In order to apply dual ascent for a particular optimization problem, the following conditions must be satisfied,

\begin{enumerate}
    \item the Lagrangian function $L(x,y^k)$ must be a strictly convex function. In case if this condition is not being obeyed, then the updating of $x$ step in the equation might end up having multiple solutions.
    
    \item The Lagrangian function $L(x,y^k)$ must be bounded below.
\end{enumerate}

\section{Dual Decomposition}
\textbf{Formal definition: }Given a linearly separable function $f(x)$ in such a way that it can be decomposed into a set of convex functions from $f(x_1), f(x_2), f(x_3, ...., f(x_n)$, then the Lagrangian function is also linearly separable on $x$. \\

If $f$ is linearly separable on $x$ as in equation \ref{DDF},

\begin{equation}
\label{DDF}
    f\left(x\right)=f_1 \left(x_1 \right)+f_2 \left(x_2 \right)+\ldotp \ldotp \ldotp \ldotp \ldotp +f_N \left(x_N \right)
\end{equation}

Then, the Lagrangian function corresponding to this function is also linearly separable on $x$ as in equation \ref{DDL}

\begin{equation}
\label{DDL}
    \begin{array}{l}
\begin{array}{l}
L\left(x,y\right)=L_1 \left(x_1 ,y\right)+L_2 \left(x_2 ,y\right)+\ldotp \ldotp \ldotp \ldotp \ldotp +L_n \left(x_n ,y\right)-y^T b\\
\mathrm{where},L_i \left(x_i ,y\right)=f_i \left(x_i \right)+y^T A_i x_i 
\end{array}
\end{array}
\end{equation}

Here, the minimization of $x$ will also split into $n$ different problems where, $x_i$ can be computed as in equation \ref{DAXA}

\begin{equation}
    \label{DAXA}
    x_i^{k+1} ={\mathrm{argmin}}_{x_i } \;L_i \left(x_i ,y^k \right)
\end{equation}

This linear separability of a single optimization problem into multiple problems enables us to perform distributed optimization of the objective function. The computation of $x$ value can thus be computed in distributed manner in multiple devices, providing $N$ devices in total for computing the values of $x_1, x_2, x_3, ...., x_N$ in separate devices. Whereas the computation of the $y$ variable or the Lagrangian multiplier takes place in the master device as per the equation \ref{DAYA},

\begin{equation}
    \label{DAYA}
    y^{k+1} =y^k +a^k \left({\sum A_i x_i }^{k+1} -b\right)
\end{equation}

The equation \ref{DAYA} resembles the equation \ref{DAY}, the only difference we can observe is in the $A.x$ where, in Dual decomposition the sum aggregate of all the $A_i x_i$ is performed.

The crux of philosophy of Distributed Computation using the Dual Decomposition Algorithm is as follows,

\begin{enumerate}
    \item Scatter $y^k$ or the vector of Lagrangian multipliers to all the slave nodes.
    \item Compute the $x_i$ vector in the individual nodes.
    \item Gather $A_i.x_i$ from all the individual nodes and update the $y^k$ in the master node
    \item Follow steps 1, 2, 3 until the solution convergence. 
\end{enumerate}

The main thing which is to be noted here that, all the disadvantages or the restrictions of the Dual ascent algorithm applies to this Algorithm also. The main disadvantage is that when the Lagrangian does not have a unique solution this algorithm will oscillate i.e., doesn't converge towards the solution.

\section{Method of Multipliers}
This algorithm solves the oscillation problem of the Dual Ascent algorithm by using Augmented Lagrangian function instead of Lagrangian function i.e., an second order proximal term is added at the end of the usual Lagrangian function. We add a penalty term to avoid the constant violation. This term is always positive, where this term drives our solution towards the point where the $A.x = b$ condition is being obeyed. Hence the Lagrangian function used in this algorithm is as in the equation \ref{MML},

\begin{equation}
    \label{MML}
    L\left(x,y\right)=f\left(x\right)+y^T \left(\mathrm{Ax}-b\right)+\frac{\rho }{2}||A\ldotp x-b||_2^2
\end{equation}

After this change the upcoming steps are as same as the Dual Ascent algorithm. Method of Multipliers is nothing but a combination of the Dual Ascent algorithm and the augmented Lagrangian function. The forthcoming steps to be performed are to initially assume a value for `y', find the next iteration `x' value as per equation \ref{MMX}, find the next iteration value of `y' as per the equation \ref{MMY}.

\begin{equation}
    \label{MMX}
    x^{k+1} ={\mathrm{argmin}}_x \;L\left(x,y^k \right)
\end{equation}

\begin{equation}
    \label{MMY}
    y^{k+1} =y^k +\rho \left({\mathrm{Ax}}^{k+1} -b\right)
\end{equation}

While solving an optimization problem with the Method of Multipliers, we basically want two conditions to hold simultaneously, which are the Primal and dual conditions of feasibility. This can also be obtained by computing the partial derivative of the Lagrangian function with respect to $y$ and $x$ respectively. These condition hold simultaneously only on the solution point. The conditions can be written mathematically as in equation \ref{MMC1} \& \ref{MMC2}
respectively.

\begin{equation}
    \label{MMC1}
    {A\ldotp x}^* -b=0
\end{equation}

\begin{equation}
    \label{MMC2}
    \nabla f\left(x^* \right)+A^T \ldotp y^* =0
\end{equation}

\subsection{How are the Primal and Dual feasibility achieved?}

We know that, $x^{k+1}$ term is computed by minimizing the Lagrangian function by using the previous iteration $x$ value, which is $x^k$, hence the gradient of Lagrangian with respect to $x$ will be zero. The dual feasibility can be be proved by using this fact as basis as in equation \ref{MMP},

\begin{equation}
    \label{MMP}
    \begin{array}{l}
\nabla_x L_p \left(x^{k+1} ,y^k \right)=0\\
\Rightarrow \nabla_x f\left(x^{x+1} \right)+A^T \left(y^k +\rho \left({\mathrm{Ax}}^{k+1} -b\right)\right)\\
\Rightarrow \nabla_x f\left(x^{k+1} \right)+A^T y^{k+1} 
\end{array}
\end{equation}

The dual update \( y^{k+1} = y^k + \rho \left( \mathrm{Ax}^{k+1} - b \right) \) makes the point \( (x^{k+1}, y^{k+1}) \) dual feasible. The primal feasibility is achieved automatically when the term \( \mathrm{Ax}^{k+1} - b \) tends to zero.

\section{ADMM}
The main advantage of ADMM or Alternating Direction Method of Multipliers is that it has inherited two advantageous properties correspondingly from two optimization algorithms discussed earlier, which is the Convergence property of Method of Multipliers and The decomposition property of Dual Decomposition.\\

The ADMM problem formulation can be described as follows, there is an objective function which is sum of two different function in terms of two different variables $x$ and $z$, where, the variables $x$ and $z$ are subject to an equality constraint, which can mathematically expressed as in equation \ref{ADMM-P}.

\begin{equation}
    \label{ADMM-P}
    \begin{array}{l}
\mathrm{objective}\;\mathrm{function}:\min_{x,z} \;f\left(x\right)+g\left(z\right)\\
\mathrm{subject}\;\mathrm{to}:A\ldotp x+B\ldotp z=c
\end{array}
\end{equation}

The Augmented Lagrangian term in for this problem can be written as in equation \ref{ADMM-L},

\begin{equation}
    \label{ADMM-L}
    L_{\rho } \left(x,z,y\right)=f\left(x\right)+g\left(z\right)+y^T \left(\mathrm{Ax}+\mathrm{Bz}-c\right)+\frac{\rho }{2}||\mathrm{Ax}+\mathrm{Bz}-c||_2^2
\end{equation}

The algorithmic steps for optimizing the above described problem using ADMM is as follows,

\begin{enumerate}
    \item Initiate `z' and `y' as a randomly generated values
    \item find the next iteration value of `x' by using the following equation
    
    \begin{equation}
        x^{k+1} ={\mathrm{argmin}}_x L_{\rho } \left(x,z^k ,y^k \right)
    \end{equation}
    
    \item find the next iteration value of `z' using the following equation,
    
    \begin{equation}
    z^{k+1} ={\mathrm{argmin}}_x L_{\rho } \left(x^{k+1} ,z,y^k \right)
    \label{DAY}
    \end{equation}

    \item find the next iteration value of `y' using the following equation,
    \begin{equation}
        y^{k+1} =y^k +\rho \left({\mathrm{Ax}}^{k+1} +{\mathrm{Bz}}^{k+1} -c\right)
    \end{equation}
\end{enumerate}

While solving an optimization problem with ADMM too, we basically want three conditions to hold simultaneously similar to the method of multipliers, which are one Primal and a couple dual conditions of feasibility. This can also be obtained by computing the partial derivative of the Lagrangian function with respect to $x, z$ and $y$ respectively.The conditions can be written mathematically as in equation \ref{ADMM1} and \ref{ADMM2} respectively.

\begin{equation}
    \label{ADMM1}
    \mathrm{Ax}+\mathrm{Bz}-c=0
\end{equation}

\begin{equation}
    \label{ADMM2}
    \begin{array}{l}
\nabla f\left(x\right)+A^T y=0\\
\nabla \;g\left(z\right)+B^T y=0
\end{array}
\end{equation}

These optimality conditions can be proved in the same manner as we did for the method of multipliers by taking the first order derivative of Lagrangian function with respect to each variable.

\section{Parallel ADMM}
Let us consider an optimization (minimization) problem of the following form as in equation \ref{PADMM-P},

\begin{equation}
    \label{PADMM-P}
    \min_x \;f\left(x\right)=\sum_{i=1}^N f_i \left(x\right)
\end{equation}

This can brought into the ADMM formulation as follows as in equation \ref{PADMM-F}

\begin{equation}
    \label{PADMM-F}
    \begin{array}{l}
\min_{x_i ,z} \;\sum_{i=1}^N f_i \left(x_i \right)\\
s\ldotp t\ldotp \;x_i -z=0,i=1,\ldotp \ldotp \ldotp \ldotp ,N
\end{array}
\end{equation}

The augmented Lagrangian function of this ADMM formulation can be written as follows as in equation \ref{PADMM-L1},

\begin{equation}
    \label{PADMM-L1}
    \begin{array}{l}
L_{\rho } \left(x_1 ,\ldotp \ldotp \ldotp \ldotp \ldotp \ldotp x_N ,z,y\right)=\sum_{i=1}^N \;L_i \\

\end{array}
\end{equation}

here, the term  $L_i$ can be expanded as follows as in equation \ref{ADMM-L2}

\begin{equation}
    \label{ADMM-L2}
     L_i =f_i \left(x_i \right)+y_i^T \left(x_i -z\right)+\frac{\rho }{2}||x_i \;-z||_2^2 
\end{equation}

Now, the iterative convergence of this solution involves the convergence of three variables $x, z$ and $y$, where the computation of next iteration value of $x$ i.e., $x_{k+1}$ can be performed in parallel using the equation \ref{PADMM-X}, after which, the updation of variable $z$ must be computed centrally as per the equation \ref{PADMM-Z}, followed by which the computation of $y_{k+1}$ by using $y_k$ must take place in a parallel as per equation \ref{PADMM-Y}.

\begin{equation}
    \label{PADMM-X}
    x_i^{k+1} ={\mathrm{argmin}}_{x_i } \left(f_i \left(x_i \right)+y_i^{\mathrm{kT}} \left(x_i -z^k \right)+\frac{\rho }{2}||x_i \;-z||_2^2 \right)
\end{equation}

\begin{equation}
    \label{PADMM-Y}
    z^{k+1} =\frac{1}{N}\sum_{i=1}^N \left(x_i^{k+1} +\frac{1}{\rho }y^k \right)
\end{equation}

\begin{equation}
    \label{PADMM-Z}
    y_i^{k+1} =y_i^k +\rho \left(x_i^{k+1} -z^{k+1} \right)
\end{equation}

\section{ALADIN}

\begin{figure}[]
\centering5    \includegraphics[width=15cm]{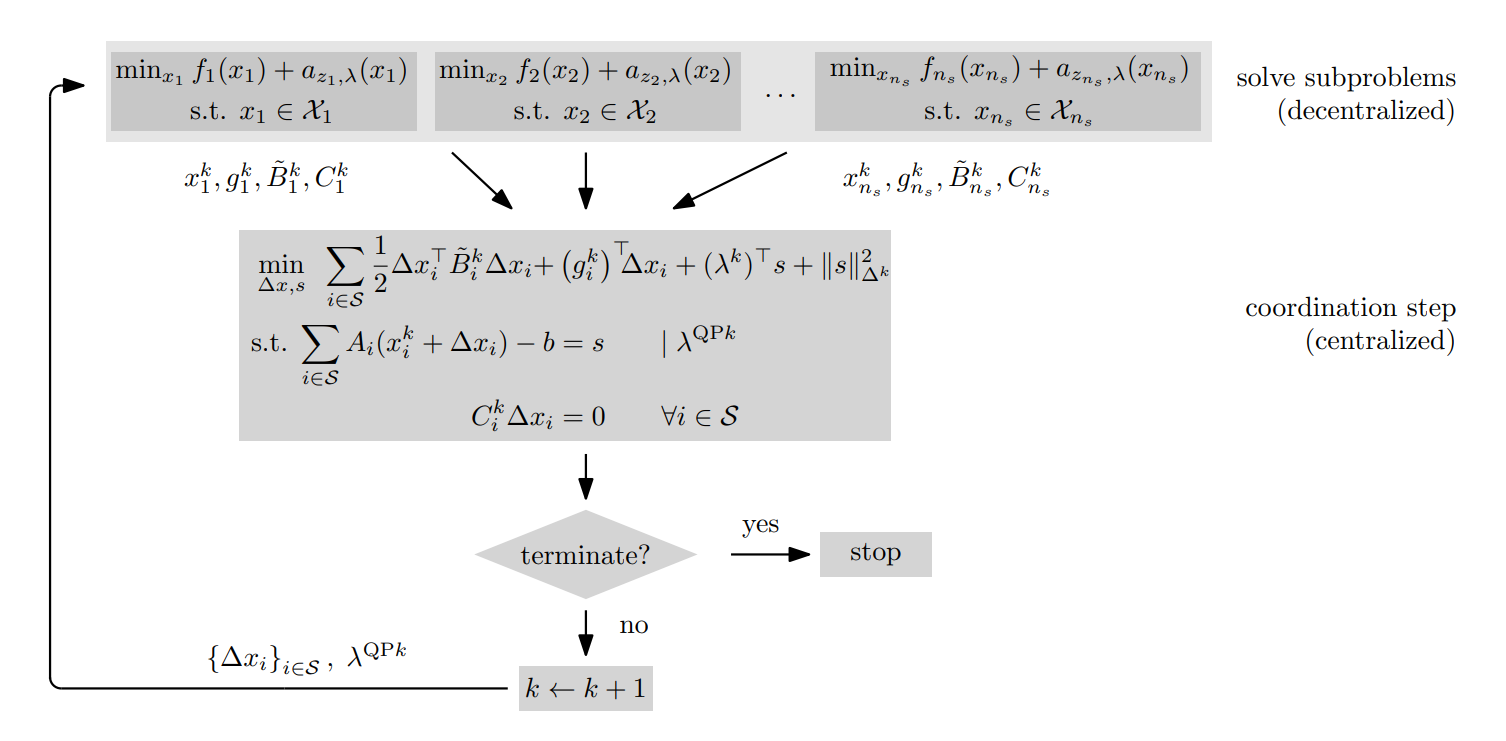}
    \caption{The results of Runtime analysis}
    \label{result}
\end{figure}

The major disadvantages of the algorithms such as ADMM, Dual decomposition and the method of multipliers is that, the convergence of these algorithms are guaranteed only in the case of convex and strictly-convex problems. One of the few algorithms which guarantees convergence for non-convex problems is ALADIN or Augmented Lagrangian Alternating Direction Inexact Newton. In ADMM we introduced an auxiliary variable $z$ into the problem formulation, whereas in ALADIN we do not do that, instead we directly deal with the partial augmented Lagrangian as in equation \ref{ALADIN-LAG} of the objective function with respect to the coupling constraints of the equation \ref{ALADIN-FORM}.

\begin{equation}
    \label{ALADIN-FORM}
    \begin{array}{l}
\min_{x_i ,\ldotp \ldotp \ldotp ,x_R } \;\sum_{i\in R} \;f_i \left(x_i \right)\\
\mathrm{subject}\;\mathrm{to}:\;g_i \left(x_i \right)=0\\
\;\;\;\;\;\;\;\;\;\;\;\;\;\;\;\;\;\;\;\;h_i \left(x_i \right)\le 0\\
\;\;\;\;\;\;\;\;\;\;\;\;\;\;\;\;\;\;\;\;\sum_{i\in R} \;A_i x_i =b
\end{array}
\end{equation}

\begin{equation}
    \label{ALADIN-LAG}
    L_{\rho } \left(x,\lambda \right)=\sum_{i\epsilon R} f_i \left(x_i \right)+l_{x_i } +\lambda^T \left(\sum_{i\epsilon R} A_i x_i -b\right)+\frac{\rho }{2}||\sum_{i\epsilon R} A_i x_i -b\;||_2^2
\end{equation}

In the above equation \ref{ALADIN-LAG} we know that the fist term corresponds to our objective functions, the third term is the penalty term corresponding to the equality constraints $A_i.x_i=b$ and the last term is our augmented Lagrangian term with $\rho$ as step size. Here, the second term of this equation is newly introduced in ALADIN which is a penalty term, where $l_{x_i}$ will take a value $1$ if the variable $x_i$ appears in the constraints, infinity in all other cases.

The further procedures are quite similar to the Method of Multipliers. We shall consider the method of multipliers, but instead of applying a full minimization of $L_\rho$ with respect to $x$ we apply only one equality-constrained SQP step yielding the following optimization formulation which happens in the master node with objective function as in equation \ref{SQP-M-O} and constraints as in equation \ref{SQP-M-C},

\begin{equation}
    \label{SQP-M-O}
    \begin{array}{l}
\min_x \;\sum_{i\epsilon R} \frac{1}{2}\Delta x_i^T B_i^k \Delta x_i +\nabla f_i^T \left(x_i \right)\Delta x_i +\lambda^{\mathrm{kT}} \left(\sum_{i\epsilon R} A_i \left(x+\Delta x_i \right)-b\right)\\
\;\;\;\;\;\;\;\;\;\;\;\;\;\;\;\;\;\;\;\;\;\;\;\;\;\;\;\;\;\;\;\;\;\;\;\;\;\;\;\;\;\;\;\;\;\;\;\;\;\;\;\;\;\;\;\;\;\;\;\;\;\;\;\;\;\;+\frac{\rho }{2}||\sum_{i\epsilon R} A_i \left(x+\Delta x_i \right)-b||
\end{array}
\end{equation}

\begin{equation}
\label{SQP-M-C}
    \begin{array}{l}
\mathrm{subject}\;\;\mathrm{to}:\tilde{g_i } \left(x_i^k \right)+\nabla \tilde{g_i } \left(x_i^k \right)\\
\mathrm{where},\tilde{g_i } {\left(x_i^k \right)}^T =\left(g_i {\left(x_i \right)}^T ,{\left(h_i {\left(x_i \right)}^T \right)}_{A\left(x_i \right)} \right)
\end{array}
\end{equation}

here, ${\left(h_i {\left(x_i \right)}^T \right)}_{A\left(x_i \right)}$ is the set of all active inequality constraints.

Here, the matrix $B_i^k$ is a positive definite approximation of the Hessian of the full Lagrangian as in equation \ref{SQP-M-L-B},

\begin{equation}
    \label{SQP-M-L-B}
    B_i^k =\nabla_{x_i x_i }^2 \left(f_i \left(x_i^k \right)+\gamma_i^T g_i \left(x_i^k \right)+\mu_i^T h_i \left(x_i^k \right)\right)
\end{equation}

For the multiplier update we apply the standard dual ascent step from the method of multipliers as in the equation \ref{SQP-M-LAM}

\begin{equation}
    \label{SQP-M-LAM}
    \lambda^{k+1} =\lambda^k +\alpha^k \left(\sum_{i\epsilon R} A_i x_i^{k+1} -b\right)
\end{equation}

In principle, one could apply this algorithm now to equality-constrained problems. This would yield a very effective algorithm since if $\rho$ is large enough, the QP in equation \ref{SQP-M-O} becomes strongly convex and thus it can be replaced by solving the KKT conditions which is a linear system of equations. However, if inequality constraints are present, the question arises how to obtain the active set $A(x_i)$. An alternative is to consider inequality constraints in lagrangian function itself, but this would make the equation \ref{SQP-M-O} substantially more difficult to solve, since the KKT
conditions also entail inequality constraints in this case.\\

ALADIN uses a different approach, it introduces a local NLP step very similar to step 1 of ADMM. This step reads,

\begin{equation}
\label{11}
    \min_{x_i } \;\sum_{i\epsilon R} \tilde{f_i } \left(x_i \right)+\lambda^{\mathrm{kT}} A_i x_i +\frac{\nu }{2}||x_i -z_i^k ||_{\sum_i }^2
\end{equation}

where $\sum_i$ is a (usually diagonal) positive definite scaling matrix and where we introduce auxiliary variables $z_i$ serving as a second iterate in ALADIN. As an active set for \ref{SQP-M-O}, one can use the active set from the minimization
of problem in equation \ref{11}.

\subsection{SQP algorithm in detail}
One of the most effective methods for non-linearly constrained optimization generates steps by solving quadratic sub-problems. This sequential quadratic programming (SQP) approach can be used both in line search and trust-region frameworks, and is appropriate for small or large problems. Unlike linearly constrained Lagrangian methods, which are effective when most of the constraints are linear, SQP methods show their strength when solving problems with significant non-linearity in the constraints.

\subsubsection{SQP formulation}
Let us consider an equality constrained optimization problem as in equation \ref{SQP-O},

\begin{equation}
    \label{SQP-O}
    \begin{array}{l}
\mathrm{minimize}\;f\left(x\right)\\
\mathrm{subject}\;\mathrm{to}\;c\left(x\right)=0
\end{array}
\end{equation}

where, $f$ is a function from $R^n$ to $R$ and $c$ is a function from $R^n$ to $R^m$. In this problem our main aim is to optimize our objective function by iteratively converging the value of $x$ and simultaneously the value of Lagrangian multipliers corresponding to the constraints. The simplest derivation of SQP methods, which we present now, views them as an application of Newton’s method to the KKT optimality conditions for equation \ref{SQP-O}.\\
Here, we know that the Lagrangian function of this problem can be written as in equation \ref{SQP-L},

\begin{equation}
    \label{SQP-L}
    L\left(x,\lambda \right)=f\left(x\right)-\lambda^T c\left(x\right)
\end{equation}

 We use A(x) to denote the Jacobian matrix of the constraints, that is, as in equation \ref{SQP-J},
 
 \begin{equation}
    \label{SQP-J}
    {A\left(x\right)}^T =\left\lbrack \nabla c_1 \left(x\right),\nabla c_2 \left(x\right),\ldotp \ldotp \ldotp \ldotp \ldotp \ldotp ,\nabla c_m \left(x\right)\right\rbrack
\end{equation}

where $c_i(x)$ is the $i^th$ component of the vector $c(x)$. The first-order (KKT) conditions of this equality-constrained problem can be written as a system of $n + m$ equations in the $n + m$ unknowns $x$ and $\lambda$,

 \begin{equation}
    \label{SQP-KKT}
    F\left(x,\lambda \right)=\left\lbrack \begin{array}{c}
\nabla f\left(x\right)-{A\left(x\right)}^T \lambda \\
c\left(x\right)
\end{array}\right\rbrack
\end{equation}

Any solution \( (x^*, \lambda^*) \) of this equality-constrained problem for which \( A(x^*) \) is a full rank matrix, will satisfy the equation \eqref{SQP-KKT}. One approach that suggests itself is to solve the nonlinear equations by using Newton’s method. The Jacobian of equation \eqref{SQP-KKT} can be written as shown in equation \eqref{SQP-KKT-J}.

 \begin{equation}
    \label{SQP-KKT-J}
    F^{\prime } \left(x,\lambda \right)=\left\lbrack \begin{array}{cc}
\nabla_{\mathrm{xx}}^2 L\left(x,\lambda \right) & -{A\left(x\right)}^T \\
A\left(x\right) & 0
\end{array}\right\rbrack
\end{equation}

The Newton's step for iteration from $(x_k,\lambda_k)$ can be written as follows, as in equation \ref{SQP-N},

 \begin{equation}
    \label{SQP-N}
    \left\lbrack \begin{array}{c}
x_{k+1} \\
\lambda_{k+1} 
\end{array}\right\rbrack =\left\lbrack \begin{array}{c}
x_k \\
\lambda_k 
\end{array}\right\rbrack +\left\lbrack \begin{array}{c}
p_k \\
p_{\lambda } 
\end{array}\right\rbrack
\end{equation}

here, $(p_k,p_\lambda)$ can be computed as follows as in equation \ref{SQP-P},

 \begin{equation}
    \label{SQP-P}
    \left\lbrack \begin{array}{cc}
\nabla_{\mathrm{xx}}^2 L\left(x,\lambda \right) & -{A\left(x\right)}^T \\
A\left(x\right) & 0
\end{array}\right\rbrack \left\lbrack \begin{array}{c}
p_k \\
p_{\lambda } 
\end{array}\right\rbrack =\left\lbrack \begin{array}{c}
-\nabla f\left(x\right)+{A\left(x\right)}^T \lambda \\
-c\left(x\right)
\end{array}\right\rbrack
\end{equation}

This Newton iteration is well-defined when the KKT matrix is non-singular. The KKT matrix is non-singular if the following assumption holds at \( (x, \lambda) = (x_k , \lambda_k ) \),

\begin{enumerate}
    \item The constraint Jacobian \( A(x) \) has full row rank.
    \item The matrix \( \nabla_{xx}^2L(x, \lambda) \) is positive definite on the tangent space of the constraints, that is,
    \( d^T \nabla_{xx}^2L(x, \lambda) d > 0 \) for all \( d \neq 0 \) such that \( A(x)d = 0 \).
\end{enumerate}

The first assumption is the linear independence constraint qualification, that is throughout the algorithm we assume that every constraint is independent of each other. The second
condition holds whenever $(x, \lambda)$ is close to the optimum $(x^*, \lambda^*)$.

\subsection{Alternate Way to Formulate SQP}
There is an alternative way to view the iteration. Suppose that at the iterate \( (x_k, \lambda_k) \) we model the optimization problem using the quadratic program along with Taylor's series approximation as shown in equation \eqref{ASQP-F},

\begin{equation}
    \label{ASQP-F}
    \begin{array}{l}
\min_p \;\;{\;\;f}_k +\nabla f_k^T p+\frac{1}{2}p^T \nabla_{\mathrm{xx}}^2 L_k p\\
\mathrm{subject}\;\mathrm{to}\;\;A_k p+c_k =0
\end{array}
\end{equation}

If Assumptions mentioned above holds, this problem has a unique solution $(p_k,l_k)$ that satisfies the constraints as in equation \ref{ASQP-C},
\begin{equation}
    \label{ASQP-C}
    \begin{array}{l}
\nabla_{\mathrm{xx}}^2 L_k p_k +\nabla f_k -A_k^T l_k =0\;\\
A_k p_k +c_k =0
\end{array}
\end{equation}

The vectors $p_k$ and $l_k$ can be identified with the solution of the Newton equations \ref{SQP-P}. If we subtract $A^T_k.\lambda_k$ from both sides of the first equation in \ref{SQP-P}, we obtain the following as in equation \ref{ASQP-M},

\begin{equation}
    \label{ASQP-M}
    \left\lbrack \begin{array}{cc}
\nabla_{\mathrm{xx}}^2 L_k p_k  & -A_k^T \\
A_k  & 0
\end{array}\right\rbrack \left\lbrack \begin{array}{c}
p_k \\
\lambda_{k+1} 
\end{array}\right\rbrack =\left\lbrack \begin{array}{c}
-\nabla f_k \\
-c_k 
\end{array}\right\rbrack
\end{equation}

\section{Applications of ALADIN}

\subsection{Application in Machine learning}

Here we give a simple classification example how ALADIN-$\alpha$ can be used for machine learning problems. The goal here is finding a suitable parameter $\omega$ that classify the input data into its label.

\subsubsection{Loss Function}
Let $x_j$ be the j-th input data and $y_j$ be its label, let $N$ be the number of input data, let $n_x$ be the dimension of input data, let $\omega$ be the decision variable. We consider a $l_2$-regularized logistic regression loss function as in the equation \ref{APP2-L},

\begin{equation}
    \label{APP2-L}
    \min_{x\epsilon R^d } \;f\left(\omega \right)=\frac{1}{N}\sum_{j=1}^N \log \left(1+e^{-y_i x_j^T \omega } \right)+\frac{\gamma }{2}\;||\omega ||_2^2
\end{equation}

here, we are adding the term $\frac{\gamma }{2}||\omega ||_2^2$ in order to prevent overfitting, where $\gamma$ is a hyperparameter chosen by user.

\subsubsection{Distributed Problem set-up}
To solve this problem, we divide the input data set into several groups and each group specifies a subsystem. We set the number of subsystems $N_{sub}$ to 10, that is, the capacity of each subsystem is $cap = N/N_{sub}$.

Next, we define the decision variable $\omega$ and set up the OCP problem. Because the volume of each group is $cap$ and the dimension of data point is $n_x$, the dimension of $\omega$ should be $cap*n_x$. It is obvious that we can also divide $\omega$ into $cap$ groups, we denote it as $\omega_j, j=1,2,...,cap$. Therefore, following equality constraints as in equation \ref{APP2-I} occurs naturally,

\begin{equation}
    \label{APP2-I}
    \omega_1 =\omega_2 =\ldotp \ldotp \ldotp =\omega_{\mathrm{cap}}
\end{equation}

Note that the form of the objective functions for each subsystem are same, so we define the objective function with parameter $xy$, which represents the input data and their labels for each subsystem.

\begin{verbatim}
w  = SX.sym('w', [cap*nx 1]);
xy = SX.sym('xy', [cap*(nx+1) 1]); 
ff  = 0;
gg = [];
for i = 1:cap
    j   = (i-1)*nx;
    k   = (i-1)*(nx+1);
    ff  = ff + 1/N*log(1+exp(xy(k+nx+1)*xy(k+1:nx)'*w(j+1:nx))) +...
            gamma/(2*N)*w(j+1:nx)'*w(j+1:nx); % objective function
    if i > 1
        for p = 1:nx
            gg = [gg; w(p)-w(j+p)]; % equality constraint
        end
    end
end
\end{verbatim}

Next, we construct the consensus matrix ${A_i}, i=1,2,...,N_{sub}$

\begin{verbatim}
eyebase = eye(nx*cap);
zerobase = zeros(nx*cap);
AA{1} = repmat(eyebase, Nsubs-1, 1);
for i = 2:Nsubs
    AA{i} = [repmat(zerobase,i-2,1);eyebase;repmat(zerobase,Nsubs-i,1)];
end
\end{verbatim}

In the last step, we convert the CasADi symbolic expressions to the MATLAB functions and set up the initial guess $z_i^0$ and $\lambda_0$. Note that the objective function is parameterized with input data $xy$ and the constraints for each subsystem are the same.

\begin{verbatim}
for i = 1:Nsubs
    ML.locFuns.ffi{i} = Function(['f' num2str(i)], {w,xy}, {ff}); 
    ML.locFuns.ggi{i} = Function(['g' num2str(i)], {w,xy}, {gg});
    ML.locFuns.hhi{i} = Function(['h' num2str(i)], {w,xy}, {[]});
    ML.AA{i} = AA{i};
    ML.zz0{i} = zeros(nx*cap,1);
    ML.p{i}   = reshape(set((i-1)*cap + 1:cap, :)', [], 1);
end
\end{verbatim}

\subsubsection{Solution using ALADIN-$\alpha$}
To solve this distributed problem with ALDIN-$\alpha$, we still need to set up some options.

\begin{verbatim}
opts.rho = 1e3;
opts.mu = 1e4;
opts.maxiter = 10;
opts.term_eps = 0; 
opts.plot = 'true';
sol_ML = run_ALADIN(ML,opts);
\end{verbatim}

If the option ``plot" is ``true", we can see the figure \ref{result2} which shows that algorithm converges in about 5 iterations, which is quite fast.

\begin{figure}[]
\label{result2}
\centering
    \includegraphics[width=13cm]{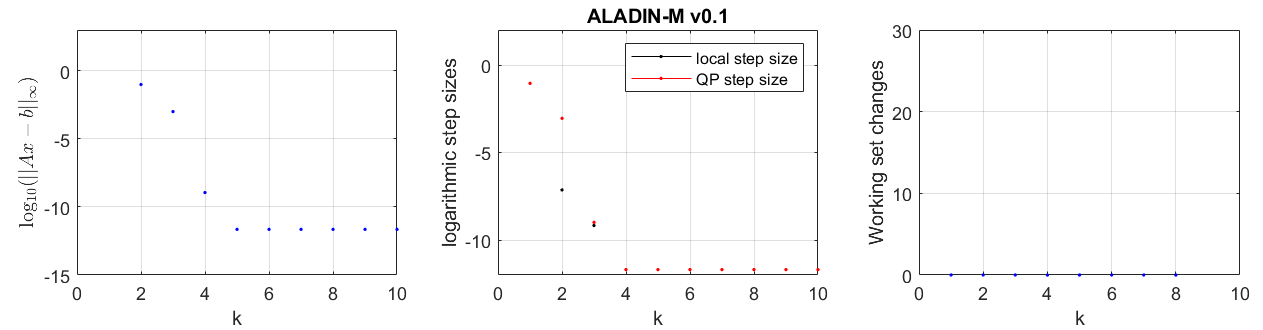}
    \caption{The graph that shows convergence of the solution}
    \label{result}
\end{figure}

\subsection{Sensor Localization}
Here we consider a sensor network localization problem from the SIAM ALADIN paper. We illustrate, how the ``parfor" option of ALADIN-$\alpha$ can be used for parallel execution. For this example the MATLAB parallel computing toolbox is required.

\subsubsection{Problem Set-up}
\subsubsection{Problem Set-up}
Let \( N \) be the number of sensors, and let \( X_i = (x_i, y_i)^T \in \mathbb{R}^2 \) be the unknown position of the \( i \)-th sensor. Let \( \eta_i \) be the estimated position, and let \( \zeta_i \) be the position of sensor \( i+1 \) as estimated by sensor \( i \). The measurement error is given by \( \eta_i - X_i \) and is assumed to be Gaussian distributed with variance \( \sigma_i^2 I_{2 \times 2} \). We further denote the measured distance between sensor \( i \) and sensor \( i+1 \) by \( \overline{\eta}_i \). If we define the decision variable as \( x_i = (X_i^T, \zeta_i^T) \in \mathbb{R}^4 \), then we can formulate the overall problem as shown in equation \eqref{APP1-O},

\begin{equation}
    \label{APP1-O}
    \begin{array}{l}
\min_x \sum_{i=1} f_i \left(x_i \right)\\
s\ldotp t\ldotp {\;\;\;\;\;h}_i \left(x_i \right)\le 0\;\;\;\;\;\forall \;\;\;i\epsilon \left\lbrace 1,\ldotp \ldotp \ldotp \ldotp \ldotp \ldotp \ldotp ,N\right\rbrace \\
\;\;\;\;\;\;\;\;\;\;\;\;\;\xi_i =x_{i+1} \;\;\;\;\;\;\;\forall \;\;\;i\epsilon \left\lbrace 1,\ldotp \ldotp \ldotp \ldotp \ldotp \ldotp \ldotp ,N\right\rbrace \\
\mathrm{with},\\
f_i \left(x_i \right)=\frac{1}{4\sigma_i^2 }||X_i -n_i ||_2^2 \;+\frac{1}{4\sigma_{i+1}^2 }||\xi_i \;-\eta_{i+1} ||_2^2 +\frac{1}{2\sigma_{i+1}^2 }{\left(||X_i -\xi_i ||_2^2 -\bar{\eta_i } \right)}^2 \\
h_i \left(x_i \right)={\left(||X_i -\xi_i ||_2^2 -\bar{\eta_i } \right)}^2 -\overline{\sigma_i^2 } 
\end{array}
\end{equation}

\subsubsection{Implementation}
For the implementation, firstly the problem needs to be defined in a way that is compatible to ALADIN-$\alpha$. The definitions of variables and functions are executed in separate functions: for the computation of $\eta_i$ and $\overline\eta_i$ the sensors are assumed to be equidistantly located in a circle, as in equation \ref{AAP1-I}

\begin{equation}
    \label{APP1-I}
    X_i =\left(N\;\cos \left(\frac{2i\pi }{N}\right),N\;\sin \left(\frac{2i\pi }{N}\right)\right)
\end{equation}

The measurement errors are assumed to be normal distributed with variance sigma. The neighbour of sensor $n$ is assumed to be sensor $1$, thus we define $\eta_{n+1} = \eta_1$. We can code the same as follows,
\\
\begin{verbatim}

function [eta,eta_bar] = getEta(N, d, sigma)
 eta     = zeros(d, N + 1);
 eta_bar = zeros(1, N);
 for i = 1 : N
     eta( :, i ) = [N * cos(2 * i * pi / N) + normrnd(0, sigma) ; ...
                    N * sin(2 * i * pi / N) + normrnd(0, sigma)];
     eta_bar(i)  = 2 * N * sin( pi / N) + normrnd(0, sigma);
 end
 eta(:, N + 1) = eta(:, 1);
end

\end{verbatim}

Implementing $f$ and $h$ as above we obtain,

\begin{verbatim}
function [F] = getObjective(N, y, eta, eta_bar, sigma)
F = zeros(N, 1);
F = sym(F);
F(:) = 1/(4*sigma^2)*((y(1,:)-eta(1,1:N)).^2+(y(2,:)-eta(2,1:N)).^2)+
     1/(4*sigma^2)*((y(3,:)-eta(1,2:end)).^2+(y(4,:)-eta(2,2:end)).^2)+
     1/(2*sigma^2)*(sqrt((y(1,:)-y(3,:)).^2+(y(2,:)-y(4,:)).^2)-eta_bar(:)').^2;
end
\end{verbatim}

\begin{verbatim}
function [H] = getInequalityConstr(N, y, eta_bar)
 H = zeros(N, 1);
 H = sym(H);
 H(:) = (sqrt((y(1, :)-y(3, :)).^2 + (y(2, :)-y(4,:)).^2)-eta_bar(:)').^2;
end
\end{verbatim}

The coupling condition $\zeta_i = \eta_{i+1}$ can be formulated as $\sum A_i.x_i=0$ with $A_i$ taking values as in equation \ref{APP1-A},

\begin{equation}
    \label{APP1-A}
    \begin{array}{l}
A_1 =\left\lbrack \begin{array}{cc}
0 & I\\
0 & 0\\
0 & 0\\
\ldotp  & \ldotp \\
\ldotp  & \ldotp \\
\ldotp  & \ldotp \\
0 & 0\\
-I & 0
\end{array}\right\rbrack ,{\;\;A}_2 =\left\lbrack \begin{array}{cc}
-I & 0\\
0 & I\\
0 & 0\\
\ldotp  & \ldotp \\
\ldotp  & \ldotp \\
\ldotp  & \ldotp \\
0 & 0\\
0 & 0
\end{array}\right\rbrack ,{\;\;\;A}_3 =\left\lbrack \begin{array}{cc}
0 & 0\\
-I & 0\\
0 & I\\
\ldotp  & \ldotp \\
\ldotp  & \ldotp \\
\ldotp  & \ldotp \\
0 & 0\\
-I & 0
\end{array}\right\rbrack ,\ldotp \ldotp \ldotp \ldotp ,A_N =\left\lbrack \begin{array}{cc}
0 & 0\\
0 & 0\\
0 & 0\\
\ldotp  & \ldotp \\
\ldotp  & \ldotp \\
\ldotp  & \ldotp \\
-I & 0\\
0 & I
\end{array}\right\rbrack \\
\\
\mathrm{while}\;\mathrm{setting},\;I=\left\lbrack \begin{array}{cc}
1 & 0\\
0 & 1
\end{array}\right\rbrack 
\end{array}
\end{equation}
\\
In Matlab this can be implemented as follows:
\begin{verbatim}
function [AA] = getCouplingMatrix(N, n)
 I = [1, 0; 0, 1];
 A0 = zeros(2*N, n);
 A1 = A0;
 A1(1:2, 3:4) = I;
 A1(2*N - 1: 2*N, 1:2) = -I;
 AA(1) = mat2cell(A_1, 2 * N, n);
 for i = 2 : 1 : N
    A_i = A0;
    A_i(2*(i-2) + 1 : 2*(i-2) + 2, 1:2) = -I;
    A_i(2*i - 1: 2*i, 3:4 ) = I;
    AA(i) = mat2cell(A_i, 2*N, n);
 end
end
\end{verbatim}

A start vector can be defined similarly to the estimated positions:
\begin{verbatim}
function [zz0] = getStartValue(N, sigma)
 initial_position = zeros(2, N);
 for i = 1  : N
    initial_position(1, i) = N * cos( 2 * i * pi / N ) + normrnd(0, sigma);
    initial_position(2, i) = N * sin( 2 * i * pi / N ) + normrnd(0, sigma);
 end
 zz0 = cell(1, N);
 for i = 1 : N-1
    zz0(i) = {[initial_position(:, i); initial_position(:, i + 1)]};
 end
 zz0(N) = {[initial_position(:, N); initial_position(:, 1)]};
end
\end{verbatim}

Such that the overall problem can be set up with the following function:
\begin{verbatim}
function [sProb ] = setupSolver(N, sigma)
n = 4;             
d = 2;   
 y = sym('y%d%d', [N n], 'real');
 y = y';
 [eta, eta_bar] = getEta(N, d, sigma);          
 F   = getObjective(N, y, eta, eta_bar, sigma); 
 H   = getInequalityConstr(N, y, eta_bar);      
 AA  = getCouplingMatrix(N, n);                
 zz0 = getStartValue(N, sigma);  
sProb.llbx = cell(1, N);
sProb.uubx = cell(1, N);
for i = 1 : N
    sProb.llbx(i) = mat2cell([-inf; -inf; -inf; -inf], 4, 1);
    sProb.uubx(i) = mat2cell([ inf;  inf;  inf;  inf], 4, 1);
end
sProb.locFuns.ffi          = cell(1, N);
sProb.locFuns.hhi          = cell(1, N);
for i = 1 : N
    sProb.locFuns.ffi(i) = {matlabFunction(F(i), 'Vars', {y(:, i)})} ;
    sProb.locFuns.hhi(i) = {matlabFunction(H(i), 'Vars', {y(:, i)})} ;
end
sProb.AA = AA;
sProb.zz0 = zz0;
\end{verbatim}

\subsubsection{Runtime Analysis}

For the runtime analysis, the idea is to ŕun the sensor network localization problem with varying number of sensors both with a decentral and a central optimization step. To do so, firstly a vector with a number of sensors is needed and secondly a vector with variances. Then, the time needed for the decentral and the central optimization is measured and can be plotted.

\begin{verbatim}
N = [5, 10, 15 , 20, 25, 30, 35, 40, 50, 60, 70, 80, 90, 100];
sigma = [0.5, 1, 1.5, 2, 2.5, 2.5, 2.5, 2.5, 2.5, 2.5, 2.5, 2.5, 2.5, 2.5];
time = zeros(2, length(N));
for i = 1 : length(N)
    sProb = setupSolver(N(i), sigma(i));
    opts.parfor = 'true' 
    time_parfor = tic;
    sol = run_ALADINnew(sProb, opts);
    time(1, i) = toc(time_parfor);
    time_for = tic;
    sol = run_ALADINnew(sProb, opts);
    time(2, i) = toc(time_for);
end
figure 
plot(N, time(1, :))
title('runtime analysis')
hold on
plot(N, time(2, :))
hold off
legend('decentral optimization', 'central optimization')
\end{verbatim}
The result of this runtime analysis can be observed as in figure \ref{result}

\begin{figure*}[htbp] 
    \centering
    \includegraphics[width=0.8\textwidth]{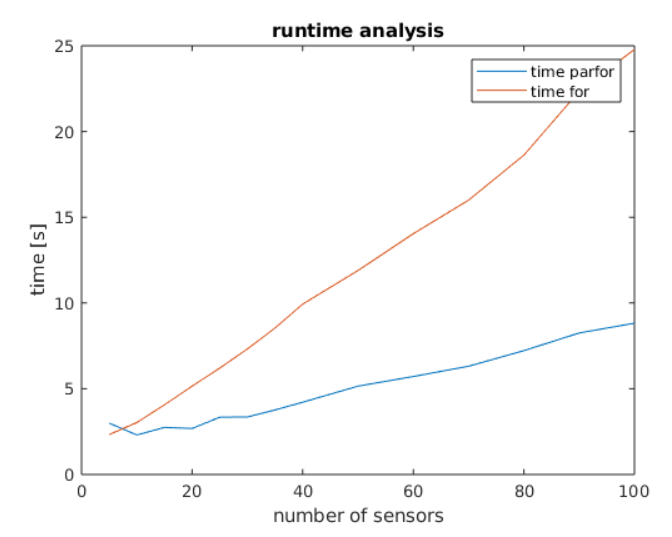} 
    \caption{The results of Runtime analysis}
    \label{result}
\end{figure*}

Thus, for the case of the senor network localization problem a significant runtime improvement can be observed when the parfor option is set. Nonetheless it needs to be mentioned that an improvement on the runtime cannot always be achieved for every problem setup using the parfor option. In general, parfor is useful when the number of local optimization problems is large and the time for solving each of the local optimization problems is relatively long.

\section{Conclusion}
Distributed optimization is a powerfull toolbox to deal with large data sets and real world problems. In this project, we explored existing methdos for distributed optimization and a state of the art method, ALADIN which guarantees convergence even for non convex objective and constraints. The learning outcomes of the project are as follows
\begin{itemize}
    \item Appreciating techniques of distributed optimization for dealing with large real world problems.
    \item Appreciating Decomposablity of functions
    \item Working knowledge of state of the art method known as ALADIN.

\end{itemize}

In future, We would like to extend the promising algorithm, ALADIN to non-convex optimization problems like Optimal Flow control, Traffic control in smart cities, continuous relaxation of combinatorial optimization, etc.

\nocite{*}
\printbibliography
\end{document}